\documentclass[12pt,twoside,reqno]{amsart}

\usepackage{amsmath,amsthm,amscd,amssymb,mathrsfs,graphicx,amsfonts,mathrsfs}
\usepackage{amsfonts}
\usepackage{amssymb,enumerate}
\usepackage{amsthm}
\usepackage[all]{xy}
\usepackage{hyperref}
\usepackage{needspace}
\hypersetup{pdftitle={Direct products of Gorenstein flat modules over coherent rings},pdfauthor={Guocheng Dai and Jiangsheng Hu}}
\allowdisplaybreaks[4] \footskip=12pt
\renewcommand{\uppercasenonmath}[1]{}

\makeatletter
\renewcommand{\section}{\@startsection{section}{1}%
  \z@{.7\linespacing\@plus\linespacing}{.5\linespacing}%
  {\normalfont\bfseries\centering\let\@secnumfont\bfseries}}
\makeatother
\numberwithin{equation}{section} \theoremstyle{plain}
\newtheorem*{thm*}{Main Theorem}
\newtheorem{thm}{Theorem}[section]

\newtheorem*{cor*}{Corollary}

\newtheorem*{lem*}{Lemma}

\newtheorem*{fact*}{Fact}

\newtheorem*{nota*}{Notation}

\newtheorem*{prop*}{Proposition}
\theoremstyle{remark}

\newtheorem*{rem*}{Remark}
\theoremstyle{plain}

\newtheorem*{observation*}{Observation}

\newtheorem*{exa*}{Example}

\newtheorem*{df*}{Definition}

\newtheorem*{con*}{Construction}

\newcommand{\Hom}{\operatorname{Hom}}

\newcommand{\D}{D_k}

\newcommand{\im}{\operatorname{Im}}

\renewcommand{\geq}{\geqslant}
\renewcommand{\leq}{\leqslant}
\renewcommand{\ker}{\Ker}

\newcommand{\GI}{\mathcal{GI}}

\newcommand{\Ext}{\operatorname{Ext}}
\renewcommand{\im}{\operatorname{im}}

\newcommand{\Gid}{\operatorname{Gid}}

\renewcommand{\ker}{\operatorname{ker}}

\theoremstyle{plain}
\newtheorem{theorem}[thm]{Theorem}
\newtheorem{lemma}[thm]{Lemma}
\newtheorem{proposition}[thm]{Proposition}
\newtheorem{corollary}[thm]{Corollary}
\newtheorem{maintheorem}[thm]{Theorem}
\newtheorem{maincorollary}[thm]{Corollary}
\theoremstyle{remark}
\newtheorem{remark}[thm]{Remark}

\newcommand{\GF}{\mathcal{GF}}
\newcommand{\PGF}{\mathcal{PGF}}
\newcommand{\Tor}{\operatorname{Tor}}
\newcommand{\Gfd}{\operatorname{Gfd}}
\newcommand{\op}{\mathrm{op}}

\begin{document}
\begin{center}
{\large\bfseries Direct products of Gorenstein flat modules over coherent rings}

\vspace{0.5cm} Guocheng Dai and Jiangsheng Hu\\
\end{center}

\bigskip
\centerline{\bfseries Abstract}
\begingroup
\leftskip10truemm \rightskip10truemm \noindent
We construct a commutative coherent local
algebra of Krull dimension zero and a countably generated projectively
coresolved Gorenstein flat module whose countable power has infinite
Gorenstein flat dimension.
This answers negatively the question of whether Gorenstein flat modules
over a right coherent ring are closed under direct products.
As a consequence, we obtain a pure injective Ding injective module
whose character module has infinite Gorenstein flat dimension and whose
character bidual has infinite Gorenstein injective dimension.
We also give equivalent conditions for Gorenstein flat modules over a
right coherent ring to be closed under direct products, in terms of relative
Mittag--Leffler conditions, stationarity, and Gorenstein flat and injective
dimensions of character biduals.
For the constructed ring, the abstract elementary class of Gorenstein
flat modules with the strong submodule relation has finite character.
A decreasing countable sequence of strong submodules of a free module
has an intersection of infinite Gorenstein flat dimension, so the class
does not admit intersections.

\vspace{0.3cm}
\par
\endgroup

\section{Introduction}\label{sec:introduction}\label{sec:scope}

Throughout, $R$ denotes an associative ring with identity, and all
modules are unital. We write $R^{\op}$ for the opposite ring of $R$.
The categories of left and right $R$-modules are denoted by
$R\text{-Mod}$ and $R^{\op}\text{-Mod}$, respectively.
Unless stated otherwise, Gorenstein flat modules are $R$-modules and
Gorenstein injective modules are $R^{\op}$-modules.
Chase proved that $R$ is right coherent if and only if direct products of
flat $R$-modules are flat (see \cite[Theorem 2.1]{Chase}).
Gorenstein flat modules were introduced by Enochs, Jenda and Torrecillas
(see \cite{EJT}). We denote their class by $\GF(R)$.
To establish closure of Gorenstein flat modules under
extensions and the existence of Gorenstein flat covers,
\v{S}aroch and \v{S}\v{t}ov\'i\v{c}ek introduced the class $\PGF(R)$
of projectively coresolved Gorenstein flat (PGF) modules
(see \cite[Section 4]{SS}).
By definition, $\PGF(R)\subseteq\GF(R)$.
The definitions of Gorenstein flat and PGF modules are recalled in
Subsection~\ref{sec:gorenstein-modules}.

\v{S}aroch and \v{S}\v{t}ov\'i\v{c}ek proved that
$\GF(R)$ is closed under direct products if and only if it is definable,
and that either condition implies right coherence of $R$
(see \cite[Proposition 4.13]{SS}; see also \cite[Paragraph 2.20]{Bird}).
Immediately before that proposition, they left open whether
$\GF(R)$ is closed under direct products whenever $R$ is right coherent.

For commutative rings, this question is part of the classification problem
posed by Christensen, Foxby and Holm (see \cite[Problem 4.12]{Survey}).
For commutative noetherian rings, closure under direct products holds in the presence
of a dualizing complex (see \cite[Theorem 5.7]{CFH}). It also holds when
every localization at a nonmaximal prime is Gorenstein
(see \cite[Proposition 6.21]{MS}). These sufficient conditions are recorded
in \cite[Theorem 4.13]{Survey}.
Closure under direct products is also related to the existence of Gorenstein flat
preenvelopes, as explained in \cite{ELR,HJ,Iacob2026}.

We show that right coherence does not ensure closure of $\GF(R)$
under direct products, even for commutative local rings of Krull dimension zero.
Our example consists of such a ring and a countably generated PGF
module whose countable power has infinite Gorenstein flat dimension.
Let $k$ be a field, and let
\begin{equation}\label{eq:main-ring}
 R=k[x_j,y_j\mid j\ge0]/
       (x_j^2,x_jy_j,y_j^2\mid j\ge0).
\end{equation}
For a module $M$ over a $k$-algebra, we write
$\D(M)=\Hom_k(M,k)$ for its $k$-linear dual.
For an $R$-module $M$ and an $R^{\op}$-module $J$, we write $\Gfd_RM$
and $\Gid_{R^{\op}}J$ for their Gorenstein flat and Gorenstein
injective dimensions, respectively.

\begin{maintheorem}\label{thm:main-counterexample}
The ring $R$ in \eqref{eq:main-ring} is commutative, coherent,
local, nonnoetherian, and of Krull dimension zero.
There is a countably generated PGF $R$-module $B$ such that, for the injective
$R^{\op}$-module $I=\D(R)$, one has
$\Tor_n^R(I,B^{\mathbb N})\ne0$ for every $n\ge1$.
Consequently, $\Gfd_R B^{\mathbb N}=\infty$.
There are a countably generated free $R$-module $Q$ and an
endomorphism $\delta:Q\to Q$ such that
$B=\ker\delta=\im\delta$, $\delta^2=0$, and the periodic complex
$(Q,\delta)$, with $Q$ in every degree and differential $\delta$,
remains exact after tensoring with every injective $R^{\op}$-module.
\end{maintheorem}

Write $R_N=\bigotimes_{j=0}^{N}(k[x_j,y_j]/(x_j,y_j)^2)$,
where the tensor products are over $k$.
Then $R=\bigcup_{N\ge0}R_N$.
Gorenstein flat $R_N$-modules are closed under direct products
for every $N$ by Corollary~\ref{cor:finite-rings-union},
whereas Gorenstein flat $R$-modules are not.
Since $R$ is nonnoetherian, this example does not settle
whether Gorenstein flat modules over arbitrary commutative
noetherian rings are closed under direct products.

We obtain two commuting differentials from the proof of
\cite[Theorem 5.7]{KMOT} by taking local algebras with
two-dimensional square-zero radicals. We combine these differentials
using the block matrix in \cite[Proposition 5.6]{KMOT}.
For each $j\ge1$, we use the factors at indices $0$ and $j$,
verify exactness after tensoring with injective modules
in Lemma~\ref{lem:finite-complete}, and extend scalars to $R$.

After tensoring the product complex with $I=\D(R)$, we construct a cycle
whose corresponding cycle over each $R_N$ is not a boundary.
Since a boundary equation involves only finitely many balancing relations,
these finite-stage obstructions also rule out a boundary over $R$.
The periodic flat resolution then gives nonzero Tor
in every positive degree.

We also obtain a counterexample in character duality from
Theorem~\ref{thm:main-counterexample}. For a module $M$, we write
$M^+=\Hom_{\mathbb Z}(M,\mathbb Q/\mathbb Z)$ for its character module.
Ding injectivity strengthens Gorenstein injectivity by testing exactness
against all FP-injective modules. We write $\GI(R^{\op})$ for the class
of Gorenstein injective $R^{\op}$-modules. The definitions are recalled
in Subsection~\ref{sec:gorenstein-modules}.

\begin{maincorollary}\label{thm:main-duality}
Let $R$ be the ring and $B$ the module in
Theorem~\ref{thm:main-counterexample}. Let
$E=B^+$. Then $E$ is pure injective and Ding injective, but
$\Gfd_R E^+=\infty$ and $\Gid_{R^{\op}}E^{++}=\infty$.
\end{maincorollary}

Thus coherence does not imply that the character module of a Ding
injective module is Gorenstein flat. The hypothesis that the character
module of every Ding injective module is Gorenstein flat is used in
\cite[Theorem 5]{Iacob2026}.

We next characterize closure of $\GF(R)$ under direct products for a
right coherent ring $R$ by requiring every direct product of at most
$|R|+\aleph_0$ Gorenstein flat modules to have finite Gorenstein flat dimension.
Using the Tor description of Gorenstein flat modules
(see \cite{SS,Holm}) and the general theory of relative
Mittag--Leffler modules and stationarity (see \cite{CI,AHH}),
we also give criteria involving a fixed first syzygy,
stationarity against a single module, and the character biduals of fixed modules.

For a class $\mathcal X$ of $R^{\op}$-modules, set
$$
\begin{aligned}
 \mathcal X^\top
 &=\{M\in R\text{-Mod}:\Tor_1^R(X,M)=0
     \text{ for every }X\in\mathcal X\},\\
 \mathcal X^\perp
 &=\{J\in R^{\op}\text{-Mod}:\Ext^1_{R^{\op}}(X,J)=0
     \text{ for every }X\in\mathcal X\}.
\end{aligned}
$$
For a single module $T$, we write $T^\top$ and $T^\perp$
for $\{T\}^\top$ and $\{T\}^\perp$, respectively.
Let $R$ be a right coherent ring and let $\rho=|R|+\aleph_0$.
There is an $R^{\op}$-module $T$ such that
$\GF(R)=T^\top$ and $\GI(R^{\op})=T^\perp$, by
\cite[Theorem 5.6 and Corollary 5.7]{SS}.
Let $L$ be the first syzygy in a fixed free presentation of $T$.
An $R^{\op}$-module $L$ is called \emph{$\GF(R)$-Mittag--Leffler} if the natural map
$L\otimes_R\prod_iG_i\to\prod_i(L\otimes_RG_i)$ is injective for
every family in $\GF(R)$ (see \cite[Definition 1.1]{AHH}).
Choose a set $\mathcal S$ of representatives of the isomorphism
classes of modules in $\GF(R)$ whose underlying sets have
cardinality at most $\rho$. Define
$C_U=\bigoplus_{H\in\mathcal S}H^{(\rho)}$.
The module $E_U=C_U^+$ is pure injective and Gorenstein injective.
Write $L=\bigcup_\lambda L_\lambda$ as the directed union of its
finitely generated submodules, ordered by inclusion.
The module $L$ is called \emph{$E_U$-stationary} if, for each $\lambda$,
the images of the restriction maps
$\Hom_{R^{\op}}(L_\mu,E_U)\to\Hom_{R^{\op}}(L_\lambda,E_U)$
are constant for all sufficiently large $\mu$
(see \cite[Definition 3.7]{AHH}).
Definable classes and Gorenstein flat preenvelopes are recalled in
Subsection~\ref{sec:character-modules}.

\begin{maintheorem}\label{thm:criteria}
Let $R$ be a right coherent ring. Fix $T$, its free presentation,
and $\mathcal S$ as above. The following conditions are equivalent.
\begin{enumerate}
\item The class $\GF(R)$ is closed under arbitrary direct products.
\item Every direct product of at most $\rho$ Gorenstein flat modules has
finite Gorenstein flat dimension.
\item The module $L$ is $\GF(R)$-Mittag--Leffler.
\item The module $L$ is $E_U$-stationary.
\item The module $C_U^{++}$ is Gorenstein flat.
\item The module $E_U^{++}$ has finite Gorenstein injective dimension.
\item The class $\GF(R)$ is definable.
\item Every $R$-module has a Gorenstein flat preenvelope.
\end{enumerate}
\end{maintheorem}

A few remarks on Theorem~\ref{thm:criteria} are in order.
First, the condition~(2) requires finite Gorenstein flat dimension for each
direct product, with no uniform bound assumed.
The argument giving a uniform bound in \cite[Proposition 2.4]{CI}
is combined here with detection of nonzero Tor on at most $\rho$ coordinates and
dimension shifting to reduce the bound to zero.
If Gorenstein flat modules are not closed under direct products,
then some direct product of at most $\rho$ Gorenstein flat modules
has infinite Gorenstein flat dimension, by the equivalence of
conditions~(1) and~(2). No large-cardinal hypothesis is needed.
This failure is realized by a
countable power of a countably generated PGF module in
Theorem~\ref{thm:main-counterexample}.

Second, \cite[Theorem 2.1 and Corollary 2.2]{CI} gives the general
relative Mittag--Leffler criterion for hereditary Tor-pairs, while
\cite[Proposition 8.13(1)]{AHH} relates relative Mittag--Leffler
conditions to strict stationarity.
Here conditions~(3) and~(4) concern the fixed syzygy $L$.
For the ring in Theorem~\ref{thm:main-counterexample}, $L$ is not
$E_U$-stationary, although $E_U$ is pure injective and Gorenstein
injective. Thus pure injectivity of $E_U$ does not ensure stabilization
of the restriction images defining stationarity.
Conditions~(5) and~(6) reduce closure under arbitrary direct
products to conditions on the character biduals of the fixed modules $C_U$ and $E_U$.

Last, the equivalences of condition~(1) with conditions~(7) and~(8)
are known: see \cite[Corollary 2.7]{CI} and
\cite[Proposition 4.13]{SS} for condition~(7), and
\cite[Corollary 3.5(c)]{RS} and \cite[p.~78]{ELR} for
condition~(8). For the ring in Theorem~\ref{thm:main-counterexample},
$\GF(R)$ is not definable, and $B^{\mathbb N}$ has neither a
Gorenstein flat preenvelope nor a PGF preenvelope.

We also give an intersection criterion by viewing a direct product
as a quotient of a free module and using the strong submodule relation
for roots of $\Ext$ (see \cite{BET,GT}).
Let $R$ be a right coherent ring and let $M,N\in\GF(R)$.
We write
$M\prec N$ and call $M$ a \emph{strong submodule} of $N$ if
$M\subseteq N$ and $N/M\in\GF(R)$, following \cite[Definition 1.2]{BET}.
Proposition~\ref{prop:tail-equivalence} characterizes closure under
direct products by an intersection condition for decreasing families
obtained as inverse images of tails of direct products.
If $(\GF(R),\prec)$ admits intersections, then $\GF(R)$ is closed
under direct products, by \cite[Lemma 10.30(c)]{GT}.
For the ring $R$ in \eqref{eq:main-ring}, the associated abstract elementary
class has finite character, but a decreasing countable family of
strong submodules of a free module has an intersection of infinite
Gorenstein flat dimension,
by Corollary~\ref{cor:countable-intersections}.

The paper is organized as follows.
In Section~\ref{sec:preliminaries}, we recall the definitions and
character-module facts used in the proofs.
In Section~\ref{sec:construction}, we construct the counterexample
and prove Theorem~\ref{thm:main-counterexample} and
Corollary~\ref{thm:main-duality}.
In Section~\ref{sec:criteria}, we give criteria for closure of
Gorenstein flat modules under direct products over right coherent
rings and prove Theorem~\ref{thm:criteria}.
Finally, in Section~\ref{sec:aec}, we give an intersection
criterion for closure of Gorenstein flat modules under direct
products and show that $(\GF(R),\prec)$ need not admit intersections.

\section{Preliminaries}\label{sec:preliminaries}

We recall the definitions and character-module facts used below.
The ring $R$ is arbitrary unless right coherence is explicitly assumed.

\subsection{Gorenstein flat and injective modules}\label{sec:gorenstein-modules}

An $R$-module is called \emph{Gorenstein flat} if it is a cycle of an exact
complex of flat $R$-modules which remains exact after tensoring with
every injective $R^{\op}$-module (see \cite{EJT}).
If the terms of this complex can be chosen projective,
the module is called \emph{projectively coresolved Gorenstein flat}
(PGF), following \cite[Section 4]{SS}.

An $R^{\op}$-module is called \emph{Gorenstein injective} if it is a cycle
of an exact complex of injective $R^{\op}$-modules which remains exact
under $\Hom_{R^{\op}}(J,-)$ for every injective $R^{\op}$-module $J$
(see \cite[Definition 2.1]{Holm}).
An $R^{\op}$-module $M$ is called \emph{FP-injective} if
$\Ext^1_{R^{\op}}(N,M)=0$ for every finitely presented $R^{\op}$-module $N$.
An $R^{\op}$-module is called \emph{Ding injective} if it is a cycle of
an exact complex of injective $R^{\op}$-modules which remains exact
under $\Hom_{R^{\op}}(M,-)$ for every FP-injective $R^{\op}$-module $M$
(see \cite[Definition 2.1]{MD} and
\cite[Definitions 3.1 and 3.2]{Gillespie}).
Ding injective modules are Gorenstein injective.

If $G$ is Gorenstein flat and $J$ is an injective $R^{\op}$-module,
then $\Tor_n^R(J,G)=0$ for every $n\ge1$, directly from the
definition (see \cite[Definition 3.1]{Holm}).
Dimension shifting therefore gives $\Tor_n^R(J,M)=0$ for $n>d$
whenever $\Gfd_RM\le d$.

\subsection{Character modules and purity}\label{sec:character-modules}

For a module $M$, we write $M^+=\Hom_{\mathbb Z}(M,\mathbb Q/\mathbb Z)$
for its \emph{character module}. The character module of an $R$-module
is an $R^{\op}$-module, and conversely.

Since $\mathbb Q/\mathbb Z$
is an injective cogenerator of abelian groups, $(-)^+$ is an exact
contravariant functor and $M^+=0$ if and only if $M=0$.
Character modules are pure injective, and the character dual
of a pure short exact sequence splits
(see \cite[Propositions 9.1 and 9.2]{Stenstrom}).

\begin{lemma}\label{lem:character-Tor}\label{lem:product-into-bidual}
Let $T$ be an $R^{\op}$-module, $M$ an $R$-module, and $q\ge0$.
Then there are natural isomorphisms
$\Tor_q^R(T,M)^+\cong\Ext^q_{R^{\op}}(T,M^+)
\cong\Ext_R^q(M,T^+)$.
A pure monomorphism $N\to M$ induces a monomorphism
$\Tor_q^R(T,N)\to\Tor_q^R(T,M)$.
For $q\ge1$, the class
$\mathcal A_q(T)=\{M:\Tor_q^R(T,M)=0\}$ is closed under pure
submodules, pure quotients, direct summands, and direct limits.
For every family $(M_i)$ of $R$-modules, there is a pure
monomorphism $\prod_iM_i\to(\bigoplus_iM_i)^{++}$.
\end{lemma}

\begin{proof}
The isomorphisms follow from \cite[p.~241]{Lambek}, applied to $R$ and $R^{\op}$.
If $0\to N\to M\to D\to0$ is pure, its character dual splits by
\cite[Proposition 9.1]{Stenstrom}. Thus the dual of
$\Tor_q^R(T,N)\to\Tor_q^R(T,M)$ is a split epimorphism,
so this Tor map is a monomorphism. The splitting also gives
$M^+\cong D^+\oplus N^+$, so the Tor--Ext isomorphism gives
closure under pure submodules and pure quotients.
Closure under direct summands follows from additivity of Tor.
Closure under direct limits follows because Tor commutes with
direct limits.
The canonical inclusion
$\bigoplus_iM_i^+\to\prod_iM_i^+=(\bigoplus_iM_i)^+$
is pure. Its character dual is therefore a split epimorphism
$(\bigoplus_iM_i)^{++}\to\prod_iM_i^{++}$.
Composing a section with the product of the canonical pure
monomorphisms $M_i\to M_i^{++}$ gives the required pure monomorphism.
Mehdi and Prest use this embedding in the proof of $(ii)\Rightarrow(iv)$
in \cite[Theorem 4.4]{MP}.
\end{proof}

Let $R$ be a right coherent ring. An $R$-module $M$ is Gorenstein flat
if and only if $M^+$ is Gorenstein injective
(see \cite[Theorem 3.6]{Holm}). This is also equivalent to $M^+$ being
Ding injective (see \cite[Lemma 2.8]{MD} and
\cite[Proposition 3.13]{Gillespie}).

A class of $R$-modules is called \emph{definable} if it is closed under direct products, direct limits,
and pure submodules (see \cite[Theorem 3.4.7]{Prest}).
For an $R$-module $M$, a map $M\to G$ with $G\in\GF(R)$ is called a
\emph{Gorenstein flat preenvelope} if every map from $M$ to a
Gorenstein flat module factors through it (see \cite[Definition 6.1.1]{EJ}).

\section{The counterexample and its consequences}\label{sec:construction}

We first construct the ring and the periodic complex and prove
the nonvanishing of Tor asserted in
Theorem~\ref{thm:main-counterexample}.
We then derive consequences for character modules, definability,
and preenvelopes.

\subsection{The coherent ring and the periodic complex}
\label{sec:periodic-complex}

Retain the field $k$ from the introduction, and let
$A_j=k[x_j,y_j]/(x_j,y_j)^2$ for $j\ge0$.
Each $A_j$ has basis $1,x_j,y_j$.
All tensor products without a subscript in this subsection are over $k$,
and empty tensor products are $k$.
Then the ring in \eqref{eq:main-ring} is the directed union
$R=\bigcup_{N\ge0}R_N$, where $R_N=\bigotimes_{0\le j\le N}A_j$.
The inclusions $R_N\to R_{N+1}$ send each element to its tensor
product with $1$.

\begin{lemma}\label{lem:explicit-ring}
The ring properties asserted in
Theorem~\ref{thm:main-counterexample} hold for $R$.
It is free over every $R_N$ for $N\ge0$ and over every
$A_0\otimes A_j$ for $j\ge1$.
\end{lemma}
\begin{proof}
The monomials obtained by choosing at most one of $x_j,y_j$ for
each $j$, with only finitely many indices used, form a $k$-basis
of $R$. Those using only indices greater than $N$ form an
$R_N$-basis of $R$. For $j\ge1$, those using neither index $0$
nor index $j$ form an $A_0\otimes A_j$-basis.
The monomial $k$-basis of $R$ is countably infinite, so $|R|=|k|+\aleph_0$.
Every finite matrix over $R$ has its entries in some $R_N$.
Since $R_N$ is finite-dimensional, the kernel of this matrix over $R_N$ is
finitely generated. Flat base change identifies the kernel over $R$
with the scalar extension of that kernel. Thus every finite matrix over $R$
has finitely generated kernel, so $R$ is coherent by the matrix
criterion (see \cite{Chase,GT}).

Let $\mathfrak m$ be the augmentation ideal of $R$.
Every element of $\mathfrak m$ is nilpotent in some $R_N$, whereas
an element with nonzero constant coefficient is a unit in a
subring $R_N$ containing it. Hence $\mathfrak m$ is the unique maximal ideal and
the nilradical. It follows that $R$ has Krull dimension zero.
The classes of $x_j,y_j$ are linearly independent in
$\mathfrak m/\mathfrak m^2$. Thus $\mathfrak m$ is not finitely
generated, and $R$ is not noetherian.
\end{proof}

We next construct a periodic complex over $A_0\otimes A_j$
for each $j\ge1$.
Let $F_0=\bigoplus_{n\ge0}A_0e_n$ and
$V_j=\bigoplus_{n\ge0}A_jf_n$ for $j\ge1$.
Define $d:F_0\to F_0$ and $t_j:V_j\to V_j$ by
$$
 d(e_{2n})=x_0e_n,\; d(e_{2n+1})=y_0e_n,\;
 t_j(f_n)=x_jf_{2n}+y_jf_{2n+1}.
$$
Both maps square to zero because $(x_0,y_0)^2=(x_j,y_j)^2=0$.
Set $\Lambda_j=A_0\otimes A_j$ and
$U_j=(F_0\otimes V_j)^{\oplus2}$. Define $\Delta_j:U_j\to U_j$ by
$$
 \Delta_j=
 \begin{pmatrix}
  d\otimes1&1\otimes t_j\\
  1\otimes t_j&-d\otimes1
 \end{pmatrix}.
$$
The elements $(e_n\otimes f_m,0)$ and $(0,e_n\otimes f_m)$,
for $n,m\ge0$, form a $\Lambda_j$-basis of $U_j$.
Thus $U_j$ is countably generated and free.
Since $d\otimes1$ and $1\otimes t_j$ commute and each squares
to zero, multiplication of the block matrices gives $\Delta_j^2=0$.

\begin{lemma}\label{lem:finite-complete}
The periodic complex $(U_j,\Delta_j)$ is exact and remains exact
after tensoring with every injective $\Lambda_j^{\op}$-module.
\end{lemma}
\begin{proof}
Define a $k$-linear map $h:F_0\to F_0$ by
$h(e_n)=0$, $h(x_0e_n)=e_{2n}$, and
$h(y_0e_n)=e_{2n+1}$.
Then $dh+hd=1$, and $h\otimes1$ commutes with $1\otimes t_j$.
Consequently,
$$
 \Delta_j\begin{pmatrix}h\otimes1&0\\0&-h\otimes1\end{pmatrix}
 +\begin{pmatrix}h\otimes1&0\\0&-h\otimes1\end{pmatrix}\Delta_j=1.
$$
Thus $(U_j,\Delta_j)$ is exact.

Let $\epsilon_j,a_{x_j},a_{y_j}$ be the basis of $\D(A_j)$
dual to $1,x_j,y_j$.
In $\D(A_j)\otimes_{A_j}V_j$, we write $af_n$ for $a\otimes_{A_j}f_n$
when $a\in\D(A_j)$.
The differential on this tensor product sends
$a_{x_j}f_n$ to $\epsilon_jf_{2n}$,
$a_{y_j}f_n$ to $\epsilon_jf_{2n+1}$, and
$\epsilon_jf_n$ to zero.
It has a $k$-linear contraction $h'$ defined by
$h'(\epsilon_jf_{2n})=a_{x_j}f_n$,
$h'(\epsilon_jf_{2n+1})=a_{y_j}f_n$, and
$h'(a_{x_j}f_n)=h'(a_{y_j}f_n)=0$.
Since $\Lambda_j$ is finite-dimensional, one has
$$
 \D(\Lambda_j)\otimes_{\Lambda_j}U_j
 \cong\bigl((\D(A_0)\otimes_{A_0}F_0)
       \otimes(\D(A_j)\otimes_{A_j}V_j)\bigr)^{\oplus2}.
$$
Under this identification, the differential $1\otimes\Delta_j$
has the same block matrix, with the maps induced by $d$ and $t_j$.
The map $1\otimes h'$ contracts the differential induced by $t_j$
and commutes with the differential induced by $d$.
Hence
$$
 (1\otimes\Delta_j)
 \begin{pmatrix}0&1\otimes h'\\1\otimes h'&0\end{pmatrix}
 +\begin{pmatrix}0&1\otimes h'\\1\otimes h'&0\end{pmatrix}
 (1\otimes\Delta_j)=1.
$$
This proves exactness after tensoring with $\D(\Lambda_j)$.

Finally, every injective $\Lambda_j^{\op}$-module $J$ is a direct
summand of a direct sum of copies of $\D(\Lambda_j)$.
Indeed, $a\mapsto[b\mapsto ab]$ embeds $J$ into
$\Hom_k(\Lambda_j,J)\cong\D(\Lambda_j)\otimes J$.
The embedding splits by injectivity. A vector-space basis of $J$
expresses the latter module as a direct sum of copies of
$\D(\Lambda_j)$. Taking direct sums and direct summands shows that
$(U_j,\Delta_j)$ remains exact after tensoring with every injective $J$.
Note that the contractions need only be $k$-linear.
\end{proof}

Extend scalars by defining $Q_j=R\otimes_{\Lambda_j}U_j$ and
$\delta_j=1\otimes\Delta_j$. Set $B_j=\ker\delta_j$.
Since $R$ is free over $\Lambda_j$, one has
$B_j=\im\delta_j$.
An injective $R^{\op}$-module restricts to an injective $\Lambda_j^{\op}$-module: restriction is right adjoint to the exact
induction functor $-\otimes_{\Lambda_j}R$.
Therefore each $B_j$ is PGF by Lemma~\ref{lem:finite-complete}.
The module $Q_j$ is countably generated and free, and $B_j$
is countably generated because $B_j=\im\delta_j$.
We identify $U_j$ with its image under $q\mapsto1\otimes q$ in $Q_j$.

For explicit formulas over $R$, write $E_{n,m}=e_n\otimes f_m$
in $F_0\otimes V_j$.
Let $c_n=x_0$ for even $n$ and $c_n=y_0$ for odd $n$.
Then
$$
 \begin{aligned}
 \delta_j(E_{n,m},0)
  &=\bigl(c_nE_{\lfloor n/2\rfloor,m},\,
           x_jE_{n,2m}+y_jE_{n,2m+1}\bigr),\\
 \delta_j(0,E_{n,m})
  &=\bigl(x_jE_{n,2m}+y_jE_{n,2m+1},\,
           -c_nE_{\lfloor n/2\rfloor,m}\bigr).
 \end{aligned}
$$

We are now in a position to prove Theorem~\ref{thm:main-counterexample}.
\begin{proof}[{\bf Proof of Theorem~\ref{thm:main-counterexample}}]
We first prove that $\Tor_n^R(I,\prod_jB_j)\ne0$ for every $n\ge1$.
The periodic flat resolution constructed below will reduce this to finding
one cycle that is not a boundary.
Let $a_y,\epsilon_0\in\D(A_0)$ extract the $y_0$-coefficient
and the constant coefficient, respectively.
Then $a_yx_0=0$. For every $f\in\D(A_0)$,
$f x_0=f(x_0)\epsilon_0$ and
$f y_0=f(y_0)\epsilon_0$.
Define $\eta=a_y\otimes e_0$ in
$\D(A_0)\otimes_{A_0}F_0$ and define the linear functional
$\lambda(\sum_na_n\otimes e_n)=a_0(y_0)$.
Then $\eta$ is a cycle, $\lambda(\eta)=1$, and
$\lambda(1\otimes d)=0$.

Let $w_j=x_jf_0\in V_j$.
For $b=\sum_n b_nf_n$, let
$\mu_j(b)=[x_j]b_0-[y_j]b_1$, where $[x_j]$ and $[y_j]$ mean
coefficient extraction in the basis $1,x_j,y_j$.
Note that $t_jw_j=0$ and $\mu_j(w_j)=1$.
In $t_j(b)$, the two extracted coefficients are both the constant
coefficient of $b_0$. Thus $\mu_jt_j=0$.
Since $\lambda$ and $\mu_j$ vanish on the images of the
differentials induced by $d$ and $t_j$, respectively,
the cycle $u_j=(\eta\otimes w_j,0)$ in
$\D(A_0)\otimes_{A_0}U_j$ and the functional
$\ell_j=(\lambda\otimes\mu_j,0)$ satisfy
\begin{equation}\label{eq:restricted-detection}
 \ell_j(u_j)=1,\; \ell_j(1\otimes\Delta_j)=0.
\end{equation}
Hence $u_j$ is not a boundary.

The natural isomorphism $\Hom_{R^{\op}}(M,\D(R))\cong\Hom_k(M,k)$
for right $R$-modules $M$ shows that $I=\D(R)$ is injective.
Define $i\in I$ by
$i(a\otimes s)=a_y(a)\epsilon(s)$ for $a\in A_0$ and
$s\in\bigotimes_{j\ge1}A_j$, where $\epsilon$ is the augmentation
of $\bigotimes_{j\ge1}A_j$.
Let $q_j=(e_0\otimes w_j,0)\in Q_j$,
$W=\prod_{j\ge1}Q_j$, and $\partial=\prod_j\delta_j$.
Then
$$
 z=i\otimes(q_j)_j\in I\otimes_RW.
$$
Since $\delta_jq_j=x_0q_j$ and $ix_0=0$, $z$ is a cycle.
The same formula gives a cycle $z_N$ in $I\otimes_{R_N}W$ for
every $N$.

Fix $N$. We detect $z_N$ in a coordinate $j>N$, so that tensoring
over $R_N$ does not impose balancing over $A_j$.
Write $C_N=\bigotimes_{1\le r\le N}A_r$ and
$S_{N,j}=\bigotimes_{r>N,\ r\ne j}A_r$.
Then $R=A_0\otimes C_N\otimes A_j\otimes S_{N,j}$ and
$Q_j=U_j\otimes C_N\otimes S_{N,j}$ after reordering factors.
For $f\in I$ and $c\in C_N$, let
$f_c(a)=f(a\otimes c\otimes1\otimes1)$ for $a\in A_0$.
Define
\begin{equation}\label{eq:finite-restriction}
 \begin{aligned}
 r_{N,j}:I\otimes_{R_N}Q_j&\longrightarrow
                         \D(A_0)\otimes_{A_0}U_j,\\
 f\otimes(q\otimes c\otimes s)&\longmapsto
                         \epsilon(s)(f_c\otimes_{A_0}q).
 \end{aligned}
\end{equation}
Balancing over $C_N$ multiplies $c$ in the same way on both sides.
Balancing over $A_0$ uses $f_c a\otimes q=f_c\otimes aq$.
Thus the map is well defined.
The map $r_{N,j}$ acts componentwise on the two summands and
commutes with every block of the differential, since each block
acts only on $U_j$. Note that $r_{N,j}$ is defined for every
$f\in\D(R)$.

Let $p_j:W\to Q_j$ be the coordinate projection.
The composite $\ell_jr_{N,j}(1\otimes p_j)$ kills every boundary
and takes value one on $z_N$ by \eqref{eq:restricted-detection}.
Hence $z_N$ is not a boundary for any $N$.

It follows that $z$ is not a boundary over $R$.
Indeed, let $V_0=I\otimes_kW$, and let $H_N$ be the span of
$(fr)\otimes w-f\otimes rw$ with $f\in I$, $w\in W$, and $r\in R_N$.
Then $I\otimes_{R_N}W=V_0/H_N$.
The subspace generated by all $R$-balancing relations is
$\bigcup_NH_N$, since each finite linear combination involves
only finitely many scalars from $R$.
If $z=(1\otimes\partial)b$ over $R$, choose a lift
$\widetilde b\in V_0$ of $b$. Then
$i\otimes(q_j)_j-(1\otimes_k\partial)\widetilde b\in H_N$ for some $N$,
so $z_N$ is a boundary, a contradiction.
The argument applies to arbitrary elements of $I$ and $W$,
including tuples in $W$ with infinitely many nonzero coordinates.

To compute the Tor groups, let $P=\prod_jB_j$ and take
direct products of the exact sequences
$0\to B_j\to Q_j\to B_j\to0$. This gives
$0\to P\to W\xrightarrow{\overline\partial}P\to0$.
Since $R$ is coherent, $W$ is flat by \cite[Theorem 2.1]{Chase}.
Thus the repeated maps $\partial$ give a flat resolution
$$
 \cdots\xrightarrow{\partial}W\xrightarrow{\partial}W
       \xrightarrow{\overline\partial}P\longrightarrow0.
$$
The cycle $z$, which is not a boundary, shows that, for every $n\ge1$,
$$
 \Tor_n^R(I,P)\cong
 \frac{\ker(1\otimes\partial)}{\im(1\otimes\partial)}\ne0.
$$
Hence $\Gfd_RP=\infty$ by the Tor vanishing recalled in
Subsection~\ref{sec:gorenstein-modules}.

It remains to construct $B$ and obtain the assertions for its
countable power. The ring properties follow from Lemma~\ref{lem:explicit-ring}.
Define $Q=\bigoplus_jQ_j$, $\delta=\bigoplus_j\delta_j$, and
$B=\bigoplus_jB_j$. Then $Q$ is countably generated and free,
$\ker\delta=\im\delta=B$, and the periodic complex remains exact
after tensoring with every injective $R^{\op}$-module. Thus $B$ is PGF and countably generated.
The coordinatewise inclusions $B_j\to B$ and projections $B\to B_j$
split $P=\prod_jB_j$ off $B^{\mathbb N}$.
Applying $\Tor_n^R(I,-)$ to this splitting gives
$\Tor_n^R(I,B^{\mathbb N})\ne0$ for all $n\ge1$.
The same Tor vanishing gives the dimension assertion.
\end{proof}

\begin{remark}\label{rmk:tensor-comparison}
Note that $q_j=x_j(E_{0,0},0)$ and $ix_j=0$.
Thus every coordinate image of $z$ is zero, but $z$ is a nonzero
cycle which is not a boundary. In particular, the comparison map
$I\otimes_R\prod_jQ_j\to\prod_j(I\otimes_RQ_j)$ is not injective.
Each coordinate vanishing uses the scalar $x_j$.
An equality in a tensor product follows from finitely many
of its defining relations, whereas $z_N$ is not a boundary over
any $R_N$, as proved after \eqref{eq:finite-restriction}.
Thus coherence does not ensure that the complex
$(W,\partial)$ remains exact after tensoring with injective modules.
\end{remark}

\begin{corollary}\label{cor:finite-rings-union}
Each ring $R_N$ in Subsection~\ref{sec:periodic-complex} has Gorenstein flat
modules closed under arbitrary direct products. The inclusions
$R_N\to R_{N+1}$ make $R_{N+1}$ a free $R_N$-module, whereas Gorenstein flat
$R$-modules are not closed under direct products.
\end{corollary}
\begin{proof}
Each $R_N$ is a commutative artinian local ring, hence complete.
It therefore admits a dualizing complex, as recalled in the proof of
\cite[Theorem 5.10]{CFH}.
Therefore Gorenstein flat $R_N$-modules are closed under direct
products by \cite[Theorem 5.7]{CFH}.
The basis of $A_{N+1}$ shows that $R_{N+1}$ is free over $R_N$.
Failure over the union follows from
Theorem~\ref{thm:main-counterexample}.
\end{proof}

\begin{remark}\label{rmk:nonmaximal-localizations}
Let $R$ be the ring in \eqref{eq:main-ring}, and let $\mathfrak m$
be its augmentation ideal. Then $\operatorname{Spec}R=\{\mathfrak m\}$,
so $R$ has no nonmaximal prime ideals. In
\cite[Proposition 6.21]{MS}, Murfet and Salarian assume that the
localization at every nonmaximal prime ideal is Gorenstein.
This localization hypothesis is vacuous for $R$.
Since $\GF(R)$ is not closed under direct products by
Theorem~\ref{thm:main-counterexample}, the noetherian hypothesis in
\cite[Proposition 6.21]{MS} cannot be replaced by coherence.
\end{remark}

\subsection{Character duality and preenvelopes}
\label{sec:duality}

We now use the nonvanishing of Tor established in the proof of
Theorem~\ref{thm:main-counterexample} to prove
Corollary~\ref{thm:main-duality}.
We then derive consequences for definability and preenvelopes.

\begin{proof}[Proof of Corollary~\ref{thm:main-duality}]
Let $P=\prod_jB_j$. By Lemma~\ref{lem:character-Tor}, there is
a pure monomorphism $P\to B^{++}$, which induces monomorphisms
$\Tor_n^R(I,P)\to\Tor_n^R(I,B^{++})$ for every $n\ge1$.
The nonvanishing of $\Tor_n^R(I,P)$
established in the proof of Theorem~\ref{thm:main-counterexample}
therefore gives $\Tor_n^R(I,B^{++})\ne0$ for every $n\ge1$.
The Tor vanishing recalled in Subsection~\ref{sec:gorenstein-modules}
gives $\Gfd_RE^+=\Gfd_RB^{++}=\infty$.

Character modules are pure injective, so $E=B^+$ is pure injective.
Since $B$ is Gorenstein flat, $E$ is Ding injective by
\cite[Lemma 2.8]{MD}.
Finally,
$\Gid_{R^{\op}}E^{++}=\Gfd_RB^{++}=\infty$ by
\cite[Proposition 3.11]{Holm}.
\end{proof}

Let $M=B^{\mathbb N}$, where $B$ is the module in
Theorem~\ref{thm:main-counterexample}.

\begin{corollary}
\label{cor:duality-failures}
\label{cor:Hom-injective}
\label{cor:no-preenvelope}
Let $R$ be the ring in \eqref{eq:main-ring}. Then neither $\GF(R)$ nor
$\GI(R^{\op})$ nor the class of Ding injective $R^{\op}$-modules is
definable.

There are a pure injective Ding injective module $E$ and an injective module
$I$ such that $\Gfd_R\Hom_{R^{\op}}(E,I)=\infty$.

The module $M$ has neither a Gorenstein flat preenvelope nor a
projectively coresolved Gorenstein flat preenvelope.
\end{corollary}

\begin{proof}
It follows from \cite[Corollary 4.6]{MP} that definable classes are
closed under character biduals.
By Corollary~\ref{thm:main-duality},
$B\in\GF(R)$ but $B^{++}\notin\GF(R)$, whereas
$E=B^+$ is Ding injective but $E^{++}$ is not Gorenstein injective.
Since every Ding injective module is Gorenstein injective,
none of the three classes is definable.

Take $E=B^+$. Since ${}_RR$ is free, $R^+$ is an injective $R^{\op}$-module. Tensor--Hom adjunction gives
$\Hom_{R^{\op}}(E,R^+)\cong E^+$, so
$\Gfd_R\Hom_{R^{\op}}(E,R^+)=\infty$ by
Corollary~\ref{thm:main-duality}.

For the final assertion, suppose that $e:M\longrightarrow G$ is a
Gorenstein flat preenvelope. Each projection
$p_j:M\longrightarrow B$ factors as $h_je$.
The map $h=(h_j):G\longrightarrow M$ then satisfies $he=1_M$.
Hence $M$ is a direct summand of a Gorenstein flat module, so it is
Gorenstein flat by \cite[Theorem 3.7]{Holm}, a contradiction.
The same proof applies to a projectively coresolved Gorenstein flat
preenvelope, since $B$ belongs to that class and that class is contained
in $\GF(R)$.
\end{proof}

\begin{remark}\label{rmk:character-duality}\label{rmk:noetherian-hypothesis}
Let $R$ be a right coherent ring. For an $R^{\op}$-module $J$,
applying the character-duality equivalence recalled in
Subsection~\ref{sec:character-modules} to $M=J^+$ shows that
$J^+$ is Gorenstein flat if and only if $J^{++}$ is Ding injective.
 For the ring in
Theorem~\ref{thm:main-counterexample},
there is a Ding injective module
$E=B^+$ with $\Gfd_RE^+=\infty$ by Corollary~\ref{thm:main-duality}.

Since Ding injective modules are Gorenstein injective,
there are a Gorenstein injective
module $E$ and an injective module $I$ over the commutative
coherent ring in \eqref{eq:main-ring} with
$\Gfd_R\Hom_{R^{\op}}(E,I)=\infty$ by
Corollary~\ref{cor:Hom-injective}.
Consequently, the question of Christensen, Foxby and Holm
(see \cite[Question 4.26]{Survey}) has a negative answer when its
noetherian hypothesis is replaced by coherence.
\end{remark}

\section{Criteria for closure of Gorenstein flat modules under direct products over right coherent rings}
\label{sec:criteria}

We now give equivalent conditions for Gorenstein flat modules
over a right coherent ring to be closed under direct products.
We first use Tor to reduce the question to direct products with
a bounded number of factors, and then express it in terms of relative
Mittag--Leffler conditions and stationarity. We complete the proof of
Theorem~\ref{thm:criteria} using character biduals, definability, and
preenvelopes. Unless stated otherwise, $R$ is a right coherent ring and
$\rho=|R|+\aleph_0$.

\subsection{Homological tests and direct products}\label{sec:tests}

The Gorenstein injective $R^{\op}$-modules form
the right class of a hereditary cotorsion pair generated by a set,
by \cite[Theorem 5.6]{SS}.
For a ring $S$ and a class $\mathcal B$ of left $S$-modules, write
${}^{\perp}\mathcal B$ for the class of left $S$-modules $M$ such that
$\Ext_S^1(M,B)=0$ for every $B\in\mathcal B$.
A pair $(\mathcal A,\mathcal B)$ of classes of $R^{\op}$-modules is called a
\emph{cotorsion pair} if $\mathcal B=\mathcal A^\perp$ and
$\mathcal A={}^{\perp}\mathcal B$. It is called \emph{hereditary} if
$\Ext^n_{R^{\op}}(A,B)=0$ for every $A\in\mathcal A$, $B\in\mathcal B$,
and $n\ge1$ (see \cite{EJ,Gillespie}).
Let $T$ be the direct sum of a set of $R^{\op}$-modules generating
this cotorsion pair.
Then $\GI(R^{\op})=T^\perp$.
An $R$-module $M$ is Gorenstein flat if and only if $M^+$ is
Gorenstein injective, as recalled in Subsection~\ref{sec:character-modules}.
By Lemma~\ref{lem:character-Tor}, we obtain
\begin{equation}\label{eq:detecting-test}
 \GF(R)=T^\top,
 \; \GI(R^{\op})=T^\perp.
\end{equation}
The first equality is the criterion in \cite[Corollary 5.7]{SS}, with
the module sides reversed. The class $\GF(R)$ is resolving, meaning that it contains
projectives and is closed under extensions and kernels of
epimorphisms between its members. It is also closed under
direct sums (see \cite[Proposition 3.2 and Theorem 3.7]{Holm}).
Closure under pure submodules, pure quotients, and direct limits
follows from Lemma~\ref{lem:character-Tor}.

We next extend \eqref{eq:detecting-test} to finite Gorenstein
flat and injective dimensions.

\begin{lemma}\label{lem:detecting-dimensions}
Let $R$ be a right coherent ring. Choose $T$ as in
\eqref{eq:detecting-test}. For every integer $d\ge0$, every $R$-module $M$,
and every $R^{\op}$-module $J$, the following equivalences hold:
$$
 \Gfd_RM\le d\ \Leftrightarrow\
 \Tor_{d+1}^R(T,M)=0,
 \;
 \Gid_{R^{\op}}J\le d\ \Leftrightarrow\
 \Ext^{d+1}_{R^{\op}}(T,J)=0.
$$
Moreover, $\Gfd_RM=\Gid_{R^{\op}}M^+$.
For each $d\ge0$, the class of $R$-modules of Gorenstein flat
dimension at most $d$ is closed under pure submodules and direct
summands.
\end{lemma}

\begin{proof}
Let $L_d$ be the $d$th syzygy in a free resolution of $M$, with
$L_0=M$. Dimension shifting gives
$\Tor_1^R(T,L_d)\cong\Tor_{d+1}^R(T,M)$.
Thus vanishing implies $L_d\in\GF(R)$ and $\Gfd_RM\le d$.
Since $\GF(R)$ is resolving, every syzygy of a Gorenstein flat
module is Gorenstein flat. Hence all positive Tor groups with $T$
vanish on $\GF(R)$. A finite Gorenstein flat resolution and its
long exact Tor sequences give the converse.

Let $J_d$ be the $d$th cosyzygy in an injective coresolution
of $J$, with $J_0=J$. Dimension shifting gives
$\Ext^1_{R^{\op}}(T,J_d)\cong\Ext^{d+1}_{R^{\op}}(T,J)$.
Hence vanishing implies $J_d\in\GI(R^{\op})$
and $\Gid_{R^{\op}}J\le d$. Conversely, heredity of the
cotorsion pair gives vanishing of all positive Ext groups from $T$
on $\GI(R^{\op})$. A finite Gorenstein injective coresolution gives the
required vanishing. The equality $\Gfd_RM=\Gid_{R^{\op}}M^+$ follows from
\cite[Proposition 3.11]{Holm}. The final assertion follows by
applying Lemma~\ref{lem:character-Tor} in degree $d+1$.
\end{proof}

The Tor statements in the following
theorem do not require coherence.

\begin{theorem}\label{thm:coordinate}
Let $R$ be a ring and let $\rho=|R|+\aleph_0$.
Let $T$ be an $R^{\op}$-module and $(M_i)_{i\in I}$ a family of $R$-modules.
Fix an integer $n\ge1$. If $\Tor_n^R(T,\prod_{i\in I}M_i)\ne0$, then there is
$X\subseteq I$ with $|X|\le\rho$ such that
$\Tor_n^R(T,\prod_{i\in X}M_i)\ne0$.
Consequently, $\mathcal A_n(T)$ is closed under arbitrary direct products
if and only if it is closed under direct products with at most $\rho$
factors.

Assume in addition that $R$ is right coherent, and let $(G_i)_{i\in I}$ be a family in
$\GF(R)$. Then $\prod_{i\in I}G_i$ is Gorenstein flat if and only
if every direct product over a subset of the index set with at most
$\rho$ factors is Gorenstein flat. In particular, if $R$ is countable, $\GF(R)$ is closed under
countable direct products if and only if it is closed under arbitrary direct products.
\end{theorem}

\begin{proof}
Fix a free resolution $F_\bullet$ of $T$ and bases of its terms.
Set $P=\prod_iM_i$. Choose a cycle representing a nonzero class,
and write it as $z=\sum_{j=1}^m e_j\otimes a_j$, where the $e_j$
are distinct basis vectors of $F_n$ and $a_j\in P$.

For each finite subset of the basis of $F_{n+1}$, write the images
of its elements using $e_1,\ldots,e_m$ and the additional basis
vectors which occur. A preimage of $z$ involving only the basis
vectors in this subset exists
exactly when a finite system of the form
\begin{equation}\label{eq:boundary-matrix}
 \varphi(\mathbf a):\;
 \exists\mathbf y\ (A\mathbf y=\mathbf a\ \text{and}\ B\mathbf y=0)
\end{equation}
has a solution, where $A$ and $B$ are finite matrices over $R$
and $\mathbf a$ is the column vector with entries $a_1,\ldots,a_m$.
The second block imposes zero coefficients on all
additional basis vectors. Every preimage of $z$ in
$F_{n+1}\otimes_RP$ involves only finitely many basis vectors of
$F_{n+1}$, so these systems characterize the boundaries whose
coefficients vanish outside $e_1,\ldots,e_m$.

Only the sizes and entries of $A$ and $B$ occur in
\eqref{eq:boundary-matrix}; the names of additional basis vectors
do not occur. There are at most $\rho$ distinct such systems.
Denote their set by $\Phi$, including the condition
$\mathbf a=0$ corresponding to the empty subset of the basis of $F_{n+1}$.
A fixed finite system is solvable in a direct product
if and only if it is solvable in every coordinate.

Since $z$ is not a boundary, every $\varphi\in\Phi$ fails in $P$.
For each $\varphi$, choose a coordinate $i_\varphi$ where it fails,
and define $X=\{i_\varphi:\varphi\in\Phi\}$.
Then $|X|\le\rho$. Coordinate projection sends $z$ to a cycle
$z_X$ in $F_n\otimes_R\prod_{i\in X}M_i$.
Every system in $\Phi$ still fails at its selected coordinate.
Hence $z_X$ is not a boundary. This proves the assertion, and the
criterion for closure under direct products follows immediately.
For the assertion about Gorenstein flat modules, apply this result
in degree one and use \eqref{eq:detecting-test}, since direct products
over subsets of the index set are direct summands.
\end{proof}

\begin{remark}\label{rmk:boundary-systems}
The finite systems in \eqref{eq:boundary-matrix} give a further criterion.
Let $R$ be a right coherent ring. Choose $T$ as in
\eqref{eq:detecting-test}. Fix a free resolution $F_\bullet$ of $T$ and
bases of its terms. For each finite subset of the basis of $F_1$, let
$\Phi$ be the associated systems of the form~\eqref{eq:boundary-matrix}.
The class $\GF(R)$ is closed under direct products if and only if, for
each such subset, one member of $\Phi$ holds for every cycle supported
on that subset over every
$G\in\GF(R)$.

If, for one such subset, no member of $\Phi$ holds for every cycle
supported on that subset over every $G\in\GF(R)$, choose for every
$\varphi\in\Phi$ a Gorenstein flat module and a tuple of coefficients
of a cycle for which $\varphi$ has no solution.
The chosen tuples determine a cycle in the direct product for which
every system in $\Phi$ has no solution. Hence the direct product of the chosen modules is not
Gorenstein flat.
Conversely, a single system valid on all coordinate cycles is
valid on the corresponding tuple in the direct product. Therefore every degree-one cycle
is a boundary, and \eqref{eq:detecting-test} applies.
\end{remark}

\begin{proposition}\label{prop:finite-product-dimensions}
Let $R$ be a right coherent ring. If every direct product of at most $\rho$
Gorenstein flat $R$-modules has finite Gorenstein flat dimension,
then $\GF(R)$ is closed under arbitrary direct products.
\end{proposition}

\begin{proof}
First, the finite dimensions have a uniform bound, as in
\cite[Proposition 2.4]{CI}. Otherwise,
for each $n<\omega$, choose such a direct product $P_n$ with
$\Gfd_RP_n>n$. After reindexing, $\prod_{n<\omega}P_n$ still has
at most $\rho$ Gorenstein flat factors. Its dimension is finite,
say $d$. Each $P_n$ is a direct summand, whence
$\Gfd_RP_n\le d$, a contradiction for $n>d$.

Let $d$ be the least uniform bound. We show that $d=0$ by adapting
the dimension-shifting argument in the proof of \cite[Theorem 5.7]{CFH}.
Suppose that $d>0$.
For an arbitrary family $(G_i)$ in $\GF(R)$ indexed by a set
of cardinality at most $\rho$, choose exact sequences
$0\to G_i\to F_i\to G_i'\to0$ from complete flat resolutions,
with $F_i$ flat and $G_i'$ Gorenstein flat.
Let $P=\prod_iG_i$ and $P'=\prod_iG_i'$.
Taking direct products gives $0\to P\to\prod_iF_i\to P'\to0$.
Since $R$ is right coherent, $\prod_iF_i$ is flat by
\cite[Theorem 2.1]{Chase}. We obtain
$\Tor_d^R(T,P)\cong\Tor_{d+1}^R(T,P')=0$
by dimension shifting and Lemma~\ref{lem:detecting-dimensions}.
Thus $\Gfd_RP\le d-1$, contradicting minimality of $d$.
Consequently $d=0$, and
$\GF(R)$ is closed under arbitrary direct products by
Theorem~\ref{thm:coordinate}.
\end{proof}

\subsection{Relative Mittag--Leffler conditions and character modules}
\label{sec:ML-criteria}

We retain $T$ from \eqref{eq:detecting-test} and express
the Tor criterion as a relative Mittag--Leffler condition
on the first syzygy of $T$.
For a class $\mathcal Q$ of $R$-modules, an $R^{\op}$-module $M$ is
$\mathcal Q$-Mittag--Leffler if the canonical map
$\tau_M:M\otimes_R\prod_iQ_i\to\prod_i(M\otimes_RQ_i)$ is
injective for every family in $\mathcal Q$;
see \cite[Definition 1.1]{AHH}.
This is the relative version of the condition introduced in
\cite[Seconde partie, Section 2]{RG}; see also \cite{HT,CI}.
Fix a free presentation
\begin{equation}\label{eq:detecting-presentation}
 0\longrightarrow L\xrightarrow{u}F_T\longrightarrow T
 \longrightarrow0.
\end{equation}

\begin{proposition}\label{prop:detecting-syzygy}
Let $R$ be a right coherent ring. For a family $(G_i)$ in $\GF(R)$, set
$P=\prod_iG_i$. Then
$\ker\tau_L=\ker(u\otimes_RP)\cong\Tor_1^R(T,P)$.
Consequently, $\GF(R)$ is closed under direct products if and only if
$L$ is $\GF(R)$-Mittag--Leffler.
\end{proposition}

\begin{proof}
The map $u\otimes_RG_i$ is injective for every $i$ by
\eqref{eq:detecting-test}. Since $F_T$ is free, $\tau_{F_T}$ is
injective. Naturality gives
$$
 \tau_{F_T}(u\otimes_RP)
 =\left(\prod_i(u\otimes_RG_i)\right)\tau_L.
$$
Since $\tau_{F_T}$ and $\prod_i(u\otimes_RG_i)$ are injective,
the two kernels agree.
The long exact Tor sequence of \eqref{eq:detecting-presentation}
identifies this kernel with $\Tor_1^R(T,P)$.
Cort\'es-Izurdiaga proves the corresponding relative Mittag--Leffler
criterion for hereditary Tor-pairs in \cite[Theorem 2.1]{CI}.
\end{proof}

We use the relative Mittag--Leffler criterion in
\cite[Theorem 5.1]{AHH}; see also \cite[Corollary 2.11]{HT}.
Let $M=\varinjlim_\lambda M_\lambda$ be a directed limit of
finitely presented $R^{\op}$-modules. 
Write $f_\lambda:M_\lambda\to M$ for the canonical maps
and $f_{\mu\lambda}:M_\lambda\to M_\mu$ for the transition maps.
Then $M$ is $\mathcal Q$-Mittag--Leffler if and only if, for each
$\lambda$, some $\mu\ge\lambda$ satisfies
\begin{equation}\label{eq:relative-ML-stages}
 \ker(f_\lambda\otimes_RQ)
 =\ker(f_{\mu\lambda}\otimes_RQ)
 \;(Q\in\mathcal Q).
\end{equation}

For the stationarity criterion, we use pure Gorenstein flat
submodules of cardinality at most $\rho$.
For a module $M$ and a subset
$X\subseteq M$, there is a pure submodule $H\subseteq M$
containing $X$ with $|H|\le\max\{|X|,|R|,\aleph_0\}$,
by \cite[Lemma 2.25(a)]{GT}.
If $R$ is right coherent and $M\in\GF(R)$, then $H\in\GF(R)$ by
Lemma~\ref{lem:character-Tor}.

Choose a set $\mathcal S$ of representatives of the isomorphism
classes of modules in $\GF(R)$ whose underlying sets have
cardinality at most $\rho$. Define
$$
 C_U=\bigoplus_{H\in\mathcal S}H^{(\rho)},
 \; E_U=C_U^+.
$$
Then $C_U$ is Gorenstein flat, and $E_U$ is pure injective and
Gorenstein injective. Note that taking $\rho$ copies of every
$H\in\mathcal S$ ensures that every direct sum of at most $\rho$
members of $\mathcal S$ is isomorphic to a direct summand of $C_U$.

For every $G\in\GF(R)$, the module $G^+$ is a direct summand of
a direct product of copies of $E_U$.
Indeed, the pure submodules $H_a$ of $G$
of cardinality at most $\rho$ form a directed family covering $G$,
by \cite[Lemma 2.25(a)]{GT}.
They are Gorenstein flat. The canonical epimorphism
$\bigoplus_aH_a\to G$ is pure: every map from a finitely presented
module to $G$ has finitely generated image contained in some $H_a$
and lifts to that summand. Thus its character dual splits,
and $G^+$ is a direct summand of $\prod_aH_a^+$.
Since every $H_a$ is isomorphic to a direct summand of $C_U$,
every $H_a^+$ is a direct summand of $E_U$. Taking the direct products
of these split maps proves the assertion.

Write $L=\bigcup_\lambda L_\lambda$ as the directed union of its
finitely generated submodules, ordered by inclusion.
Each $L_\lambda$ lies in a
finite-rank free direct summand of $F_T$ and is finitely presented
by right coherence. Let $j_\lambda:L_\lambda\to L$ and
$j_{\mu\lambda}:L_\lambda\to L_\mu$ be the inclusions.
For an $R^{\op}$-module $J$, write
$H_{j_{\mu\lambda}}(J)=\im[\Hom_{R^{\op}}(L_\mu,J)\to
\Hom_{R^{\op}}(L_\lambda,J)]$,
$H_{j_\lambda}(J)=\im[\Hom_{R^{\op}}(L,J)\to
\Hom_{R^{\op}}(L_\lambda,J)]$, and
$H_{u j_\lambda}(J)=\im[\Hom_{R^{\op}}(F_T,J)\to
\Hom_{R^{\op}}(L_\lambda,J)]$.
We use the morphism-indexed notation for $H$-subgroups from
\cite[Definition 2.1]{AHH}. These subgroups consist of the maps from
$L_\lambda$ to $J$ that extend to $L_\mu$, $L$, and $F_T$,
respectively.
In the preceding notation, the $E_U$-stationarity condition for $L$
(see \cite[Definition 3.7]{AHH}) reads as follows: for every $\lambda$,
there is $\mu\ge\lambda$ such that
$H_{j_{\nu\lambda}}(E_U)=H_{j_{\mu\lambda}}(E_U)$ for all $\nu\ge\mu$.
This definition is independent of the finitely presented direct
system used; see \cite[Theorem 4.8]{AHH}.

The relation between relative Mittag--Leffler conditions and
strict stationarity is given in \cite[Proposition 8.13(1)]{AHH}.
For the fixed modules $L$ and $E_U$, we express this relation
in terms of the restriction images defined above.

\begin{proposition}
\label{lem:stationary-intersection}
\label{prop:stationary-criterion}
For every $\lambda$, one has
$\bigcap_{\mu\ge\lambda}H_{j_{\mu\lambda}}(E_U)
=H_{j_\lambda}(E_U)=H_{u j_\lambda}(E_U)$.

The following conditions are equivalent.
\begin{enumerate}
\item The module $L$ is $\GF(R)$-Mittag--Leffler.
\item The module $L$ is $E_U$-stationary.
\item For every $\lambda$, there is $\mu\ge\lambda$ such that
$H_{j_{\mu\lambda}}(E_U)=H_{u j_\lambda}(E_U)$.
\end{enumerate}
Condition~(3) says that every map $L_\mu\to E_U$, restricted
to $L_\lambda$, extends to $F_T$. Here $\mu$ depends on $\lambda$,
but not on the map being extended.
\end{proposition}

\begin{proof}
We first identify the intersection of the restriction images and deduce
$(2)\Leftrightarrow(3)$. We then prove $(1)\Leftrightarrow(3)$
using the relative Mittag--Leffler criterion and tensor--Hom adjunction.
Let $h:L_\lambda\to E_U$ belong to each image in the intersection.
Introduce variables $x_a$ with values in $E_U$, one for every
$a\in L$. Require $x_{a+b}=x_a+x_b$ for all $a,b\in L$,
$x_{ar}=x_a r$ for all $a\in L$ and $r\in R$, and
$x_a=h(a)$ for all $a\in L_\lambda$. Every finite subsystem involves finitely many
elements of $L$. They lie in some finitely generated
$L_\mu\supseteq L_\lambda$, and an extension of $h$ to $L_\mu$
solves that subsystem. Since $E_U$ is pure injective, it is
algebraically compact by \cite[Proposition 9.2]{Stenstrom}. A solution
of the entire system therefore exists and defines an extension
$L\to E_U$.
This proves the first equality. Since $E_U\in T^\perp$,
\eqref{eq:detecting-presentation} implies that
$\Hom_{R^{\op}}(F_T,E_U)\to\Hom_{R^{\op}}(L,E_U)$ is surjective. This proves the second
equality.

By the first assertion, the intersection of
the images is $H_{u j_\lambda}(E_U)$. The images stabilize if and
only if this intersection is attained at some stage. Thus the
conditions~(2) and~(3) are equivalent.

We use tensor--Hom adjunction. For a map $f:M\to N$
of $R^{\op}$-modules and an $R$-module $G$, the image of
$\Hom_{R^{\op}}(f,G^+)$ is the annihilator of $\ker(f\otimes_RG)$ under
$\Hom_{R^{\op}}(M,G^+)\cong(M\otimes_RG)^+$.
Indeed, a character vanishing on the kernel defines a character
on the image and extends to $N\otimes_RG$, because
$\mathbb Q/\mathbb Z$ is injective. Characters separate points
of quotient groups, so equality of these annihilators is equivalent
to equality of the kernels.

Assume condition~(1). Apply the criterion
\eqref{eq:relative-ML-stages} to $L$ with $Q=C_U$. The equal kernels give
$H_{j_{\mu\lambda}}(E_U)=H_{j_\lambda}(E_U)$, proving condition~(3).

Conversely, assume condition~(3). With the same $\mu$, the extension
property passes to direct products of copies of $E_U$ by extending
coordinatewise, and to their direct summands by projection.
As shown above, every
$G^+$ with $G\in\GF(R)$ is a direct summand of a direct product of
copies of $E_U$. Hence the condition holds for every such $G^+$.
Since $G^+\in T^\perp$, one has
$H_{u j_\lambda}(G^+)=H_{j_\lambda}(G^+)$.
The annihilator observation gives
$\ker(j_\lambda\otimes_RG)=\ker(j_{\mu\lambda}\otimes_RG)$
for every such $G$. Thus \eqref{eq:relative-ML-stages} gives
condition~(1).
\end{proof}

\begin{remark}\label{rmk:pure-injectivity-stationarity}
In Proposition~\ref{lem:stationary-intersection}, pure injectivity of
$E_U$ identifies the intersection of the restriction images with
$H_{j_\lambda}(E_U)$; the equality with $H_{u j_\lambda}(E_U)$ uses
$E_U\in T^\perp$. The class $\GF(R)$ is closed under direct products
if and only if, for each $\lambda$, the images $H_{j_{\mu\lambda}}(E_U)$
stabilize at $H_{u j_\lambda}(E_U)$. If $E_U$ is $\Sigma$-pure injective,
then $L$ is $E_U$-stationary by \cite[Corollary 8.11(3)]{AHH},
and hence $\GF(R)$ is closed under direct products. Thus, for the ring
in Theorem~\ref{thm:main-counterexample}, $E_U$ is not
$\Sigma$-pure injective.
\end{remark}

\subsection{Proof of Theorem~\ref{thm:criteria}}\label{sec:biduals}

We now characterize closure of $\GF(R)$ under direct products using
the character biduals of the fixed modules $C_U$ and $E_U$.

\begin{proposition}\label{prop:canonical-bidual}
Let $R$ be a right coherent ring. Then the following conditions are equivalent.
\begin{enumerate}
\item The class $\GF(R)$ is closed under direct products.
\item The module $C_U^\rho$ is Gorenstein flat.
\item The module $C_U^{++}$ is Gorenstein flat.
\item $\Gfd_RC_U^{++}<\infty$.
\item The module $E_U^{++}$ is Gorenstein injective.
\item $\Gid_{R^{\op}}E_U^{++}<\infty$.
\end{enumerate}
\end{proposition}

\begin{proof}
We prove $(4)\Rightarrow(1)$. Choose $d<\infty$ with
$\Gfd_RC_U^{++}\le d$. We first show that every direct product of
Gorenstein flat modules has Gorenstein flat dimension at most $d$.
Otherwise, there is a nonzero
$\Tor_{d+1}^R(T,\prod_{i\in X}G_i)$ with $|X|\le\rho$,
by Lemma~\ref{lem:detecting-dimensions} and
Theorem~\ref{thm:coordinate}.
Choose a free resolution of $T$ and a cycle representing a nonzero
class in its tensor product with $\prod_{i\in X}G_i$.
For each $i\in X$, choose a pure Gorenstein flat submodule
$H_i\subseteq G_i$ of cardinality at most $\rho$ containing the
finitely many coefficients in that coordinate, by
\cite[Lemma 2.25(a)]{GT} and Lemma~\ref{lem:character-Tor}.
Since the terms of the resolution are free, the inclusion
$\prod_{i\in X}H_i\subseteq\prod_{i\in X}G_i$ induces a
monomorphism in each degree. Hence the same expression defines a
cycle after tensoring the resolution with $\prod_{i\in X}H_i$.
It is not a boundary, since its image in the resolution tensored with
$\prod_{i\in X}G_i$ is not a boundary.

Since $|X|\le\rho$, $D=\bigoplus_{i\in X}H_i$ is isomorphic
to a direct summand of $C_U$.
The module $\prod_{i\in X}H_i$ embeds
purely into $D^{++}$ by Lemma~\ref{lem:product-into-bidual},
and hence into $C_U^{++}$.
Thus $\Gfd_R(\prod_{i\in X}H_i)\le d$ by
Lemma~\ref{lem:detecting-dimensions}, a contradiction.
This uniform bound and Proposition~\ref{prop:finite-product-dimensions}
give condition~(1).

For $(2)\Rightarrow(1)$, suppose that $C_U^\rho$ is Gorenstein flat.
If $\GF(R)$ were not closed under direct products, the coordinate
projection and the choice of pure submodules used for
$(4)\Rightarrow(1)$, applied in degree one, would give a direct product $\prod_{i\in X}H_i$ with
$|X|\le\rho$ and $\Tor_1^R(T,\prod_{i\in X}H_i)\ne0$.
Each $H_i$ is isomorphic to a direct summand of $C_U$, so
$\prod_{i\in X}H_i$ is isomorphic to a direct summand of $C_U^X$,
and hence of $C_U^\rho$, a contradiction.

Assume condition~(1). Then $C_U^\rho\in\GF(R)$, proving
condition~(2).
Closure under direct products, pure submodules, and direct limits
makes $\GF(R)$ definable.
It follows from \cite[Corollary 4.6]{MP} that definable classes are
closed under character biduals. Hence condition~(3) holds, and it implies condition~(4).
Finally, $(C_U^{++})^+=E_U^{++}$, so
$(3)\Leftrightarrow(5)$ and $(4)\Leftrightarrow(6)$
by Lemma~\ref{lem:detecting-dimensions}.
\end{proof}



We end this section with the proof of Theorem~\ref{thm:criteria} in the introduction.
\begin{proof}[{\bf Proof of Theorem~\ref{thm:criteria}}]
Condition~(1) implies condition~(2), and the converse follows from
Proposition~\ref{prop:finite-product-dimensions}.
Proposition~\ref{prop:detecting-syzygy} gives $(1)\Leftrightarrow(3)$.
By Proposition~\ref{prop:stationary-criterion}, $L$ is
$\GF(R)$-Mittag--Leffler if and only if it is $E_U$-stationary,
so $(3)\Leftrightarrow(4)$.
Proposition~\ref{prop:canonical-bidual} gives
$(1)\Leftrightarrow(5)\Leftrightarrow(6)$.
Since $\GF(R)$ is closed under pure submodules and direct limits,
$(1)\Leftrightarrow(7)$ follows from the definition of a definable class
in \cite{Prest,SS}.
Finally, $(1)\Leftrightarrow(8)$ follows from
Rada and Saor\'in \cite[Corollary~3.5(c)]{RS},
since $\GF(R)$ is closed under pure submodules.
\end{proof}




\section{Abstract elementary classes and intersections}
\label{sec:aec}

We now relate direct products to intersections by taking
inverse images under epimorphisms from free modules.
Throughout this section, $R$ is a right coherent ring and
$\rho=|R|+\aleph_0$.
We use the strong submodule relation introduced in
Section~\ref{sec:introduction} (see \cite[Definition 1.2]{BET}).

An \emph{abstract elementary class} of modules is a class $\mathcal K$
together with a strong submodule relation $\prec$ satisfying the axioms
in \cite[Definition 1.1]{BET}.
The least infinite cardinal $\kappa$, at least the cardinality of the
language, such that every subset $X$ of $N\in\mathcal K$ is contained
in some $M\prec N$ with $|M|\leq |X|+\kappa$ is called its
\emph{L\"owenheim--Skolem number}.

An abstract elementary class is said to \emph{admit intersections} if every intersection
of strong submodules of a fixed member is strong in that member
(see \cite[Definition 10.29(d)]{GT}). It is said to have \emph{finite character} if, for $M\subseteq N$
in the class, the inclusion is strong whenever every finite subset $X$
of $M$ admits strong embeddings $f_X:M\longrightarrow D_X$ and
$g_X:N\longrightarrow D_X$ into a common member with
$f_X|_X=g_X|_X$ (see \cite[Definition 2.8]{ST}).

Following \cite[Section 10.3]{GT}, we write
${}^{\perp_\infty}\{T^+\}$ for the class of $R$-modules $M$
such that $\Ext_R^n(M,T^+)=0$ for every $n\ge1$.
By \eqref{eq:detecting-test}, the Tor--Ext isomorphism in
Lemma~\ref{lem:character-Tor}, and dimension shifting,
one has $\GF(R)={}^\perp\{T^+\}={}^{\perp_\infty}\{T^+\}$.
Since $T^+$ is pure injective, $(\GF(R),\prec)$ is an abstract
elementary class with L\"owenheim--Skolem number at most $\rho$,
by \cite[Theorem 10.28]{GT}. It has finite character by
\cite[Theorem 2.11]{ST}.
For the cardinal bound, let $X\subseteq G\in\GF(R)$.
Choose a pure submodule $H\subseteq G$ containing $X$
with $|H|\le |X|+\rho$.
Since $\GF(R)$ is closed under pure submodules and pure quotients,
one has $H\prec G$.
Direct sums give joint embedding, and disjoint amalgamation follows
from \cite[Lemma 2.1 and Remark 2.2]{BET}.
If $(\GF(R),\prec)$ admits intersections, then $\GF(R)$ is closed under
arbitrary direct products, by \cite[Lemma 10.30(c)]{GT}.

For the ring $R$ and the modules $B$ and $I$ in
Theorem~\ref{thm:main-counterexample}, we now construct a
decreasing sequence of strong submodules of a free module.
Here $\rho=|k|+\aleph_0$.
Let $M=B^{\mathbb N}$.
For each integer $n\ge0$, let $M_n\subseteq M$ consist of the tuples whose
first $n$ coordinates are zero. Let $F$ be the free module with one
basis vector $[m]$ for each $m\in M$, and let
$\pi:F\longrightarrow M$ send $[m]$ to $m$.
Set $N_n=\pi^{-1}(M_n)$.

\begin{corollary}
\label{cor:countable-intersections}
The modules $N_n$ form a decreasing sequence with $N_n\prec F$
for every $n$. Its intersection $N=\bigcap_nN_n$ has infinite Gorenstein
flat dimension. Thus $(\GF(R),\prec)$ does not admit intersections.
It has finite character and L\"owenheim--Skolem number $\rho=|k|+\aleph_0$.
\end{corollary}

\begin{proof}
Since $F/N_n\cong M/M_n\cong B^n$, closure under finite direct products and the
resolving property imply $N_n\prec F$. Moreover,
$N=\ker\pi$. For every $q\geq1$, the exact sequence
$0\longrightarrow N\longrightarrow F\longrightarrow M\longrightarrow0$
gives $\Tor_q^R(I,N)\cong\Tor_{q+1}^R(I,M)\ne0$.
Thus $\Gfd_RN=\infty$. Finite character and the upper bound $\rho$
for the L\"owenheim--Skolem number follow from
\cite[Theorem 2.11]{ST} and \cite[Theorem 10.28]{GT}, respectively.
For the lower bound, apply the L\"owenheim--Skolem property to
$\{1\}\subseteq R\in\GF(R)$. Every submodule of $R$ containing
$1$ equals $R$, so the L\"owenheim--Skolem number is at least
$|R|=\rho$.
\end{proof}

\begin{remark}\label{rmk:known-intersection-examples}
A class of roots of $\Ext$ whose associated abstract elementary class
does not admit intersections is given in \cite[Example 2.11]{BET}.
A class of projective modules that is closed under direct products but
whose associated abstract elementary class does not admit
intersections is given in \cite[Example 10.31]{GT}.
For the ring $R$ in \eqref{eq:main-ring}, failure of intersections in
$(\GF(R),\prec)$ is witnessed by a decreasing countable sequence of
strong submodules of a free module whose intersection has infinite
Gorenstein flat dimension, by Corollary~\ref{cor:countable-intersections}.
\end{remark}

We next apply the same inverse-image construction to a family
indexed by a limit ordinal.
Let $R$ be a right coherent ring, let $\lambda$ be a limit ordinal,
and let $(G_\beta)_{\beta<\lambda}$ be Gorenstein flat $R$-modules. Set $P=\prod_{\beta<\lambda}G_\beta$.
For $\alpha<\lambda$, let $P_\alpha\subseteq P$ consist of the tuples
whose coordinates below $\alpha$ are zero. Choose an epimorphism from a free module
$\pi:F\longrightarrow P$, and define $L_\alpha=\pi^{-1}(P_\alpha)$.
Then $\bigcap_{\alpha<\lambda}L_\alpha=\ker\pi$ and
$F/L_\alpha\cong\prod_{\beta<\alpha}G_\beta$.

\begin{proposition}
\label{prop:tail-equivalence}
\label{prop:large-cofinality}
The following conditions are equivalent.
\begin{enumerate}
\item The class $\GF(R)$ is closed under arbitrary direct products.
\item For every limit ordinal $\lambda$, every family
$(G_\beta)_{\beta<\lambda}$ in $\GF(R)$, and every epimorphism
$\pi:F\to P$ from a free module, $L_\alpha\prec F$ for all
$\alpha<\lambda$ implies
$\bigcap_{\alpha<\lambda}L_\alpha\prec F$.
\item The implication in condition~(2) holds whenever
$\lambda$ is an infinite cardinal with $\lambda\leq\rho$.
\end{enumerate}

In the notation above, if $\operatorname{cf}(\lambda)>\rho$ and
$L_\alpha\prec F$ for every $\alpha<\lambda$, then
$\bigcap_{\alpha<\lambda}L_\alpha\prec F$.
\end{proposition}

\begin{proof}
If $\GF(R)$ is closed under direct products, then $F/\ker\pi\cong P$
is Gorenstein flat. Thus $\ker\pi\prec F$, proving condition~(2).
Condition~(2) implies condition~(3).

For $(3)\Rightarrow(1)$, we adapt the argument in the proof of
$(1)\Rightarrow(2)$ in \cite[Lemma 2.8]{BET}.
Suppose that $\GF(R)$ is not closed under direct products.
There is a least cardinal
$\lambda\leq\rho$ such that a direct product of $\lambda$ Gorenstein flat
modules is not Gorenstein flat, by Theorem~\ref{thm:coordinate}. Since finite direct products are Gorenstein flat, $\lambda$
is infinite. Choose such a direct product and an epimorphism from a free module as above.
Minimality gives $F/L_\alpha\in\GF(R)$ for every $\alpha<\lambda$.
Then $L_\alpha\prec F$ by the resolving property.
Condition~(3) would give $\ker\pi\prec F$, and hence
$P\in\GF(R)$, a contradiction.

For the final assertion, assume that $\operatorname{cf}(\lambda)>\rho$
and $L_\alpha\prec F$ for every $\alpha<\lambda$.
Each direct product $\prod_{\beta<\alpha}G_\beta$, with $\alpha<\lambda$,
belongs to $\GF(R)$. Every subset $X\subseteq\lambda$
of cardinality at most $\rho$ is bounded in $\lambda$. Thus
$\prod_{\beta\in X}G_\beta$ is a direct summand of one of these
direct products and is Gorenstein flat.
We have $P\in\GF(R)$ by Theorem~\ref{thm:coordinate},
and hence $\ker\pi\prec F$.
\end{proof}

\begin{remark}\label{rmk:cofinality-restriction}
For the ring $R$ in \eqref{eq:main-ring}, condition~(2) of
Proposition~\ref{prop:tail-equivalence} fails at $\lambda=\omega$,
by Corollary~\ref{cor:countable-intersections}.
The implication in condition~(2) holds whenever
$\operatorname{cf}(\lambda)>\rho=|k|+\aleph_0$,
by the final assertion of Proposition~\ref{prop:large-cofinality}. Hence the cofinality restriction
cannot be omitted. If instead every $L_\alpha$ is pure in $F$, then
$F/L_\alpha\cong\prod_{\beta<\alpha}G_\beta$ is flat,
since $F$ is free. For each $\beta<\lambda$, we have
$\beta+1<\lambda$, so $G_\beta$ is a direct summand of
$F/L_{\beta+1}$ and is therefore flat.
By Chase's theorem \cite[Theorem 2.1]{Chase}, $P$ is flat, so
$\ker\pi$ is pure in $F$.
The argument using Chase's theorem concerns the submodules
$L_\alpha=\pi^{-1}(P_\alpha)$ and does not assert closure under arbitrary
intersections of pure submodules.
\end{remark}

\smallskip
{\bf Acknowledgements.} The research was supported by the National Natural Science Foundation of China (Grant Nos.~12571035 and 11901412). The authors would like to thank Jan \v{S}aroch and Jan Trlifaj for their valuable comments.

\section*{Statement on AI usage}

The authors acknowledge the use of ChatGPT during the preparation
of this manuscript. The overall
strategy of the paper, including the proof of nonvanishing of Tor
for a countable direct power, and the criteria for closure of
Gorenstein flat modules under direct products, was
developed by the authors. ChatGPT assisted with checking the block-matrix calculation
establishing a contracting homotopy after tensoring with
$\D(\Lambda_j)$ in the proof of Lemma~\ref{lem:finite-complete}.
It was also used to revise the English text.
The AI-assisted calculation was independently verified.
The authors take full responsibility for the content of the article.

\vspace{3mm}
\noindent\textbf{Guocheng Dai}\\
School of Mathematical Sciences, Sichuan Normal University, Chengdu 610066, P.R. China.\\
Email: \textsf{ggcdai@163.com}\\[1mm]
\textbf{Jiangsheng Hu}\\
School of Mathematics, Hangzhou Normal University, Hangzhou 311121, P.R. China.\\
Email: \textsf{hujs@hznu.edu.cn}\par
\end{document}